\documentclass[12pt,a4paper]{article}
\usepackage[utf8x]{inputenc}
\usepackage{amsmath}
\usepackage{amsfonts}
\usepackage{amssymb}
\usepackage{amsthm}
\usepackage{mathrsfs}
\usepackage{fancyhdr}
\usepackage{graphicx}
\usepackage[usenames,x11names]{xcolor}
\usepackage{arabtex}
\usepackage{framed}
\usepackage[top=2cm,bottom=2cm,margin=2cm]{geometry}
\usepackage{array}   %%%% Pour les fonctionalités assez évoluées d'un tableau

\usepackage[dvipdfm,pdfstartview=FitH,colorlinks=true,linkcolor=Red4,urlcolor=Red4,%
filecolor=green,citecolor=red]{hyperref}

\title{Minimum Hellinger distance estimates for a periodically time-varying long memory parameter}
	\author{\underline{Amine AMIMOUR $^1$} \and Karima BELAIDE $^1$ \and Ouagnina HILI $^2$\\
Department of Mathematics, Applied Mathematics Laboratory,\\    University of Bejaia, Bejaia Algeria. $^1$ \\
	Laboratory of Mathematics and New Technologies of Information,\\ National Polytechnic
	Institute Felix HOUPHOUET-BOIGNY\\ Yamoussoukro, P.O. Box 1093, Ivory Coast.$^2$\\
amineamimour@gmail.com $^1$ \and k\_tim2002@yahoo.fr$^1$ \and o\_hili@yahoo.fr$^2$}

\date{}

\def\EMdash{\leavevmode\hbox to 10.6mm{\vrule height .63ex depth -.59ex
    width 10mm\hfill}}
\theoremstyle{plain}
\newtheorem{theorem}{Theorem}

\newtheorem{proposition}{Proposition}

\newtheorem{assumption}{Assumption}
\newtheorem{condition}{Condition}

\begin{document}
\maketitle
\begin{abstract}
We consider a purely fractionally deferenced process driven by a periodically time-varying long memory parameter. We will build an estimate for the vector parameters using the minimum Hellinger distance estimation. The results are investigated through simulation studies.
\end{abstract}
\noindent\textbf{MSC 2010:} 37M10, 62M10, 91B84. \\
\textbf{Keywords:} Periodic ARFIMA, Minimum Hellinger distance, Fractionally process, Estimation, Time-varying long memory parameter.

\section{Introduction}
	One of the major prominent periodic invertible and causal process in time series analysis is the PARMA model (see \cite{V1985})  which generalizes the ARMA model (see \cite{BJ1970} ).
$(X_{t},t\in
%TCIMACRO{\U{2124} }%
%BeginExpansion
\mathbb{Z}
%EndExpansion
)$, is said to be PARMA model if it satisfies the difference equation
\begin{equation}
\label{Eq0}
\overset{P}{\underset{j=0}{\sum }}\phi _{i,j}X_{i+pm-j}=\overset{Q}{%
	\underset{k=0}{\sum }}\theta _{i,j}\varepsilon _{i+pm-k}, m \in \mathbb{Z},
%TCIMACRO{%
%\TeXButton{TeX field}{\makeatletter
%\renewcommand\theequation{\thesection.\arabic{equation}}
%\@addtoreset{equation}{section}
%\makeatother}}%
%BeginExpansion
\makeatletter
\renewcommand\theequation{\thesection.\arabic{equation}}
\@addtoreset{equation}{section}
\makeatother%
%EndExpansion
\end{equation}
where, for each season $i$ $(i=1,\dots,p)$, where p is the period, $P$ and $Q$ are the AR and MA orders respectively, and the coefficients satisfy $\phi_{t+p,j}=\phi _{t,j}$ for $j=1,..P$ and $\theta_{t+p,k}=\theta_{t,k}$ for $k=1,\dots,Q$. The sequence ($\varepsilon _{t})_{t\in\mathbb{Z}}$ is zero-mean and uncorrelated with finite variance $\sigma^{2}_{t}$, the variance is periodic in $t$ such that $\sigma^{2}_{t+pm}$ =$\sigma^{2}_{t}$. Troutman \cite{T1979} considered a periodic autoregressive process which is defined in (\ref{Eq0}) with null moving average order and inspected divers properties of this model by considering the associated stationary multivariate autoregressive process. Bentarzi and Hallin\cite{BH1994} investigated this idea. They consider the random variables in the periodic non stationary process as the elements of a multivariate stationary process which is an interest of periodic models, in addition to allowing the study of seasonal phenomena. It is that they can be exploited in the context of the analysis of the stationary multidimensional time series, to cut substantially the number of the parameters to estimate, especially for multivariate autoregressive processes (see \cite{P1978}, \cite{G1961} and \cite{N1982} ). However, the zero mean purely fractional AFRIMA$\left( 0,d,0\right) $ process $\left(
X_{t},t\in \mathbb{Z} \right)$ represented by 
\begin{equation}
\label{Eq1}
(1-B)^{d}X_{t}=\varepsilon _{t},
%TCIMACRO{%
%\TeXButton{TeX field}{\makeatletter
%\renewcommand\theequation{\thesection.\arabic{equation}}
%\@addtoreset{equation}{section}
%\makeatother}}%
%BeginExpansion
\makeatletter
\renewcommand\theequation{\thesection.\arabic{equation}}
\@addtoreset{equation}{section}
\makeatother%
%EndExpansion
\end{equation}
is also an extension of the ARMA model, where $B$ denotes the backward shift operator and $d$ can take any real number. Establishing a relationship between the fractional integration and long memory, this model goes back to Hosking \cite{HSK1981}, showing that the long memory process is invertible for $d >-\frac{1}{2}$ and stationary for $d <\frac{1}{2}$. Odaki \cite{O1993} noticed that this process is invertible even when $-1<d \leq-\frac{1}{2}$, these conditions concerning the univariate case. Ching-fan \cite{C2002} used a VARFIMA model to define the invertibility condition of multivariate p-dimensional stationary process in the sense of Hosking \cite{HSK1981}. In this article, we frame the model (\ref{Eq1}) with periodically time varying memory parameter $d$ with period $p$ by accommodating a long memory stationary model for the p-variate process. It is an interesting topic due to the importance in one hand of the ARFIMA  models and on the other hand the periodic phenomenon (see \cite{AB20201}, \cite{AB20202}). Furthermore, the model of Hosking has found its potential in long-term forecasting, and so it has turn into one of the basic famous parametric long memory models in the statistical literature. For this model, the parametrical estimation of the memory parameter has been widely used as in Yajima \cite{Y1985} determined the estimation of $d$ and $\sigma^2$ for the model defined in (\ref{Eq1}), using two methods, the least squares estimates and the maximum likelihood estimate by calculating the spectral density of the sequences ${\varepsilon_t}$ and by the density function of the maximum likelihood estimator respectively and assuming that this is a leading note in identification and estimation procedure of a general ARFIMA$(p,d,q)$. Gupta \cite{G1992} proposed a regression method for estimating the $d$ factor in this general model and proved that this estimator have mean square consistency and compared its performances with some known results such as Yajima \cite{Y1985} and others, concluding that this method ables to estimate $d$ and the ARMA parameters, Sowell \cite{S1992} considered this general model and derived the unconditional exact likelihood function, assuming that this method allows the simultaneous estimation of all the parameters of the model by exact maximum likelihood. Mayoral \cite{M2007} used a minimum distance by a new method for estimating the parameters of stationary and non-stationary ARFIMA$(p, d_{0}, q)$ process for $d_{0}<0,75$. The quasi-maximum likelihood approach for a non-stationary multivariate ARFIMA process is derived by Kamagate and Hili \cite{KH2013}. Kamagate and Hili \cite{KH2012} determined the minimum Hellinger distance estimate (MHDE) of a general ARFIMA model. Mbeke and Hili \cite{MH2018} extended this work to the multivariate case. They constructed an estimate for a vector parameters, so these known results may be used to infer the desired property of the model considered here. Although we deal only with the characteristics of the purely fractionally differenced process defined in (\ref{Eq1}) with periodic long memory parameter. Indeed, Amimour and Belaide \cite{AB2020} have proved recently that this model has a local asymptotic normality property. The cause of selecting the MHD method is that its estimate own exquisite robustness properties of the MHD estimators. It has been pioneered by Beran \cite{B1977} for independent and identically distributed observations, Hili \cite{H1995} extended these results to the case of dependent observations of non-linear time series. \\
The outline of our paper is as follows. Section \ref{sec1} describes the model, provides some required assumptions and carry out the indispensable propositions. Section \ref{sec2} is devoted to state the main results.
To see how the MHD method applies, we conduct some simulations study in section \ref{sec3}.
	\section{Notation and basic assumptions}
\label{sec1}
\subsection{Definition and notation}
We shall consider the periodic autoregressive fractionally integrated moving average processes $(X_{t},t\in
%TCIMACRO{\U{2124} }%
%BeginExpansion
\mathbb{Z}
%EndExpansion
)$ with period $p$, denoted here by PtvARFIMA$_{p}$ $(0,d_{t},0)$, which are proposed by \cite{AB2020}. The model is given by   

\begin{center}
	\begin{equation}
	\label{Eq2}
	(1-B)^{d_{t}}X_{t}=\varepsilon _{t}\iff (1-B)^{d_{i}}X_{i+pm}=\varepsilon
	_{i+pm},%
	%TCIMACRO{%
	%\TeXButton{TeX field}{\makeatletter
	%\renewcommand\theequation{\thesection.\arabic{equation}}
	%\@addtoreset{equation}{section}
	%\makeatother}}%
	%BeginExpansion
	\makeatletter
	\renewcommand\theequation{\thesection.\arabic{equation}}
	\@addtoreset{equation}{section}
	\makeatother%
	%EndExpansion
	\end{equation}
\end{center}
where for all $t\in 
%TCIMACRO{\U{2124} }%
%BeginExpansion
\mathbb{Z}
%EndExpansion
$, there exists $i=\left\{ 1,\dots,p\right\}$, $m\in 
%TCIMACRO{\U{2124} }%
%BeginExpansion
\mathbb{Z}
%EndExpansion
,$ such that $t=i+pm$, $d_{i}$ is the long memory parameter which varies over time, whose values lie in $(0,\frac{1}{2})$, and $\left( \varepsilon _{t},t\in 
%TCIMACRO{\U{2124} }%
%BeginExpansion
\mathbb{Z}
%EndExpansion
\right) $ is a zero mean white noise with finite variance $\sigma^{2}_{t}$, the variance is periodic in $t$ such that $\sigma^{2}_{t+pm}$ =$\sigma^{2}_{t}$.\\
	When $d_{i}>0$. The process (\ref{Eq2}) is invertible and has an infinite
autoregressive representation as follows:

\begin{center}
	\begin{equation}
	\label{Eq3}
	\varepsilon _{i+pm}=(1-B)^{d_{i}}X_{i+pm}=\underset{}{\overset{}{\overset{%
				\infty }{\underset{j=0}{\sum }}\text{ }\pi _{j}^{i}X_{i+pm-j}}},%
	%TCIMACRO{%
	%\TeXButton{TeX field}{\makeatletter
	%\renewcommand\theequation{\thesection.\arabic{equation}}
	%\@addtoreset{equation}{section}
	%\makeatother}}%
	%BeginExpansion
	\makeatletter
	\renewcommand\theequation{\thesection.\arabic{equation}}
	\@addtoreset{equation}{section}
	\makeatother%
	%EndExpansion
	\end{equation}
\end{center}

where 
\begin{center}
	$\pi _{j}^{i}=\frac{\Gamma (j-d_{i})}{\Gamma (j+1)\Gamma (-d_{i})}$.
\end{center}

$\Gamma
(.)$ is the gamma function.\\
When $d_{i}<\frac{1}{2}$, the process (\ref{Eq2}) is causal and has an infinite moving-average representation as follows:

\begin{center}
	\begin{equation}
	\label{Eq4}
	X_{i+pm}=(1-B)^{-d_{i}}\varepsilon _{i+pm}=\overset{\infty }{\underset{j=0}{%
			\sum }}\psi _{j}^{i}\varepsilon _{i+pm-j},%
	%TCIMACRO{%
	%\TeXButton{TeX field}{\makeatletter
	%\renewcommand\theequation{\thesection.\arabic{equation}}
	%\@addtoreset{equation}{section}
	%\makeatother} }%
	%BeginExpansion
	\makeatletter
	\renewcommand\theequation{\thesection.\arabic{equation}}
	\@addtoreset{equation}{section}
	\makeatother
	%EndExpansion
	\end{equation}
\end{center}

where 
\begin{center}
	$\psi _{j}^{i}=\frac{\Gamma (j+d_{i})}{\Gamma (j+1)\Gamma (d_{i})}.$
\end{center}
The convergence of infinite sums, in (\ref{Eq3}) and (\ref{Eq4}), is to be understood in the quadratic mean sense, for $i=1,\dots,p$. \\

We consider the parameters vector $d=\left(d_{1},\dots,d_{p}\right)$ a $p$-dimensional real vector $d \in \Theta$, where $\Theta$ is a compact subset of  $%
%TCIMACRO{\U{211d} }%
%BeginExpansion
\mathbb{R}
%EndExpansion
^{p}.$ We assume that we have an realization of size $n$,  $(X_{1},.....,X_{n})$ of the solution of equation (\ref{Eq2}). Suppose, for simplicity of notation
reasons, that the size $n$ is a multiple of $p$, i.e. $n=pn^{\prime }$. Let $i=1,\cdots ,p$ and $ m =0,1,\cdots ,n^{\prime }-1$ .\\
The related multivariate stationary process of the PtvARFIMA$_{p}$ $(0,d_{i},0)$ model is given by
\begin{center}
$\begin{pmatrix}
\overset{\infty }{\underset{j=0}{\sum }}\pi _{j}^{1} &  & 0 \\
& \ddots  & \vdots  \\
0 & \cdots  & \overset{\infty }{\underset{j=0}{\sum }}\pi _{j}^{p}	\end{pmatrix}
\left(
\begin{array}{c}
X_{1+pm-j} \\
X_{2+pm-j} \\
. \\
. \\
X_{p+pm-j}%
\end{array}%
\right)
=\left(
\begin{array}{c}
\varepsilon _{1+pm} \\
\varepsilon _{2+pm} \\
. \\
. \\
\varepsilon _{p+pm}%
\end{array}%
\right)=\epsilon_{m}.
%TCIMACRO{%
%\TeXButton{TeX field}{\makeatletter
%\renewcommand\theequation{\thesection.\arabic{equation}}
%\@addtoreset{equation}{section}
%\makeatother} }%
%BeginExpansion
\makeatletter
\renewcommand\theequation{\thesection.\arabic{equation}}
\@addtoreset{equation}{section}
\makeatother
%EndExpansion	
$
\end{center}
\subsection{Main assumptions and propositions }
\label{sec2}	
In order to deal with MHD estimation based on the Theorem 2 and 4 in $\cite{B1977}$ and Lemma 3.1 in $\cite{H1995}$, our starting point is to introduce further notations and assumptions that are necessary in the sequel. For more detail, we refer the reader to $\cite{MH2018}$.\\
Let $\hat{d}_{n}$ be the MHD estimate of $d$, which minimizes the Hellinger distance between $f_{n}$ and $f_{d}$. That is 
\begin{equation}
\hat{d}_{n}=\arg \underset{d\in \Theta }{\min }H_{2}(f_{n},f_{d}),
%TCIMACRO{%
%\TeXButton{TeX field}{\makeatletter
%\renewcommand\theequation{\thesection.\arabic{equation}}
%\@addtoreset{equation}{section}
%\makeatother} }%
%BeginExpansion
\makeatletter
\renewcommand\theequation{\thesection.\arabic{equation}}
\@addtoreset{equation}{section}
\makeatother
%EndExpansion
\end{equation}
where $H_{2}(f_{n},f_{d})$ is the Hellinger distance between $f_{n}$ and $f_{d}$ defined by
\begin{equation}
H_{2}(f_{n},f_{d})=\left( \int_{%%TCIMACRO{\U{211d} }%%BeginExpansion
\mathbb{R}
%EndExpansion
^{p}}|f_{n}^{\frac{1}{2}}(x)-f_{d}^{\frac{1}{2}}(x)|^{2}dx\right) ^{\frac{1}{2}},
%TCIMACRO{%
%\TeXButton{TeX field}{\makeatletter
%\renewcommand\theequation{\thesection.\arabic{equation}}
%\@addtoreset{equation}{section}
%\makeatother} }%
%BeginExpansion
\makeatletter
\renewcommand\theequation{\thesection.\arabic{equation}}
\@addtoreset{equation}{section}
\makeatother
%EndExpansion
\end{equation}
where  $f_{d}(.)$ is the theoretical probability density of ${\epsilon}_{m}$, with $f_{d}:
%TCIMACRO{\U{211d} }%
%BeginExpansion
\mathbb{R}
%EndExpansion
^{p}\rightarrow 
%TCIMACRO{\U{211d} }%
%BeginExpansion
\mathbb{R}
%EndExpansion
_{+}$ \\ $f_{n}(.)$ is the random function of $\hat{\epsilon}_{m}$ given by
\begin{equation*}
f_{n}(x)=\frac{1}{nh_{n}^{p}}\overset{n^{^{\prime }}-1}{\underset{m=0}{\sum }}K\left( \frac{x-\hat{\epsilon}_{m}}{h_{n}}\right),     \ \ \ x\in 
%TCIMACRO{\U{211d} }%
%BeginExpansion
\mathbb{R}
%EndExpansion
^{p},
\end{equation*}
where
\begin{equation}
\begin{pmatrix}
\overset{n}{\underset{j=0}{\sum }}\pi _{j}^{1} &  & 0 \\
& \ddots  & \vdots  \\
0 & \cdots  & \overset{n}{\underset{j=0}{\sum }}\pi _{j}^{p}%
\end{pmatrix}%
\left(
\begin{array}{c}
X_{1+pm-j} \\
X_{2+pm-j} \\
. \\
. \\
X_{p+pm-j}%
\end{array}%
\right)=\left(
\begin{array}{c}
\hat{\varepsilon}_{1+pm} \\
\hat{\varepsilon}_{2+pm} \\
. \\
. \\
\hat{\varepsilon}_{p+pm}%
\end{array}%
\right)=\hat{\epsilon}_{m},\  m =0,\cdots ,n^{\prime }-1,\ \ \ \ \ \ \ \ \ \ \ \ \ \ \
%TCIMACRO{%
%\TeXButton{TeX field}{\makeatletter
%\renewcommand\theequation{\thesection.\arabic{equation}}
%\@addtoreset{equation}{section}
%\makeatother} }%
%BeginExpansion
\makeatletter
\renewcommand\theequation{\thesection.\arabic{equation}}
\@addtoreset{equation}{section}
\makeatother
%EndExpansion
\end{equation}
and	$K:
%TCIMACRO{\U{211d} }%
%BeginExpansion
\mathbb{R}
%EndExpansion
^{p}\rightarrow 
%TCIMACRO{\U{211d} }%
%BeginExpansion
\mathbb{R}
%EndExpansion
_{+}$ is the kernel density function and $h_{n}$ the bandwidth.\\
Let $\tilde{f}_{n}(.)$ denote the kernel density estimation of $f_{d}$ of ${\epsilon}_{m}$, such that
\begin{equation*}
\tilde{f}_{n}(x)=\frac{1}{nh_{n}^{p}}%
\overset{n^{^{\prime }}-1}{\underset{m=0}{\sum }}K\left( \frac{x-\epsilon
	_{m}}{h_{n}}\right),   \ \ \ x\in 
%TCIMACRO{\U{211d} }%
%BeginExpansion
\mathbb{R}
%EndExpansion
^{p},
\end{equation*}

In the whole of the article, we consider the following assumptions:
\begin{assumption}
\label{Ass1}
The process given in (\ref{Eq2}) satisfies the sufficient condition of invertibility and causality with $0<d_{i}<\frac{1}{2}$.
\end{assumption}
\begin{assumption}
\label{Ass2}
$E(|\epsilon _{m}|^{t})<+\infty $ for $t\geq 1$.\\
For all $(u;v)$ $\in
%TCIMACRO{\U{211d} }%
%BeginExpansion
\mathbb{R}
%EndExpansion
^{2p}$, we have\\
$\int_{%
%TCIMACRO{\U{211d} }%
%BeginExpansion
\mathbb{R}
%EndExpansion
^{p}}K^{2}(u)du<\infty ,$ $\int_{%
%TCIMACRO{\U{211d} }%
%BeginExpansion
\mathbb{R}
%EndExpansion
^{p}}u_{i}K(u)du=0$ for $1\leq i\leq p$.\\
$\int_{%
%TCIMACRO{\U{211d} }%
%BeginExpansion
\mathbb{R}
%EndExpansion
^{p}}u_{i}u_{j}K(u)du=0,$ $\int_{%
%TCIMACRO{\U{211d} }%
%BeginExpansion
\mathbb{R}
%EndExpansion
^{p}}u_{i}^{2}K(u)du<\infty $ for $1\leq j\leq p$.\\
There exists $c>0$ such that $\underset{u\in
%TCIMACRO{\U{211d} }%
%BeginExpansion
\mathbb{R}
%EndExpansion
^{p}}{sup}|K(u+v)-K(u)|\leq c|v|.$
\end{assumption}
\begin{assumption}
\label{Ass3}
$\epsilon _{m}
$ admits a density absolutely continuous with respect to the Lebesgue measure on $%
%TCIMACRO{\U{211d} }%
%BeginExpansion
\mathbb{R}
%EndExpansion
^{p}$. For all $d\in \Theta $ and $x\in
%TCIMACRO{\U{211d} }%
%BeginExpansion
\mathbb{R}
%EndExpansion
^{p},$ the functions $x\rightarrow f_{d}(x)$ and $x\rightarrow f_{d}^{\frac{1%
}{2}}(x)$ are continuously differentiable.
\end{assumption}
\begin{assumption}
\label{Ass4}
For all $x\in
%TCIMACRO{\U{211d} }%
%BeginExpansion
\mathbb{R}
%EndExpansion
^{p}$, the functions $d\rightarrow \frac{\partial }{\partial d_{i}}f_{d}^{%
\frac{1}{2}}(x),$ for $1\leq i\leq p$ and $ d\rightarrow \frac{\partial ^{2}}{%
\partial d_{i}\partial d_{k}}f_{d}^{\frac{1}{2}}(x),$ for $1\leq j,k\leq p$,
are bounded, continuous and defined in $L^{2}(%
%TCIMACRO{\U{211d} }%
%BeginExpansion
\mathbb{R}
%EndExpansion
^{p})$.
\end{assumption}
\begin{assumption}
\label{Ass5}
$h_{n}=n^{\alpha }\ell (n),$ $-1<\alpha <0$ with $\ell(.) $ a slowly varying function,
$\underset{n\rightarrow \infty }{lim}h_{n}=0,$ $\underset{n\rightarrow
\infty }{lim}nh_{n}=\infty ,$ $\underset{n\rightarrow \infty }{lim}\frac{%
\ell (an)}{\ell (n)}=1,a>0.$
For all $d\in \Theta $, $\underset{x\in
%TCIMACRO{\U{211d} }%
%BeginExpansion
\mathbb{R}
%EndExpansion
^{p}}{sup}|\frac{\partial ^{j}f_{d}}{\partial x_{k}^{j}}(x)|$ $<\infty $, $%
j=0,1,2,....$and $k=1,......,p.$
\end{assumption}
\begin{assumption}
\label{Ass6}
For $d,d^{^{\prime }}$ $\in \Theta ,$ $d\neq d^{^{\prime }}$ implies that $%
\{x\in
%TCIMACRO{\U{211d} }%
%BeginExpansion
\mathbb{R}
%EndExpansion
^{p}/f_{d}(x)\neq f_{d^{^{\prime }}}(x)\}$ is a set of positive Lebesgue measure.
\end{assumption}
\begin{assumption}
\label{Ass7}
There exists a constant M such that $\underset{x\in
%TCIMACRO{\U{211d} }%
%BeginExpansion
\mathbb{R}
%EndExpansion
^{p}}{sup}f_{n}(x)\leq M<\infty .$
\end{assumption}
Note also that $g_{d}(x)=f_{d}^{\frac{1}{2}}(x),$ $g_{d}^{\prime }(x)=\frac{\partial g_{d}}{%
\partial _{d}}(x),$ $g_{d}^{\prime \prime }(x)=\frac{\partial ^{2}g_{d}}{%
\partial d\partial d^{t}}(x).$
$U_{d}(x)=\left[ \int_{%
%TCIMACRO{\U{211d} }%
%BeginExpansion
\mathbb{R}
%EndExpansion
^{p}}g_{d}^{\prime }(x)\left[ g_{d}^{\prime }(x)\right] ^{t}dx\right]
^{-1}g_{d}^{\prime }(x).$\\
Here and in what follows $t$ denotes the transpose, and $\overset{a.s}{\rightarrow }$ the convergence with probability one.
\begin{condition}
	\label{c1}
	The components of $g_{d}^{\prime }$ and $g_{d}^{\prime \prime }$ are in $%
	L_{2}$ and the norms of the components are continuous functions at $d$.
\end{condition}
\begin{condition}
	\label{c2}
	$\int_{%
		%TCIMACRO{\U{211d} }%
		%BeginExpansion
		\mathbb{R}
		%EndExpansion
		^{p}}g_{d}^{\prime \prime }(x)g_{d}(x)dx$ is a non-singular $%
	(p\times p)$-matrix.
\end{condition}

\begin{proposition}
\label{Pro1}
Under the assumptions \ref{Ass1}-\ref{Ass3} we have for any $p\ge 2$,
\begin{equation}
f_{n}(x)-f_{d}(x)\overset{a.s}{\rightarrow }0,\ as\ n \rightarrow \infty.
%TCIMACRO{%
%\TeXButton{TeX field}{\makeatletter
%\renewcommand\theequation{\thesection.\arabic{equation}}
%\@addtoreset{equation}{section}
%\makeatother}}%
%BeginExpansion
\makeatletter
\renewcommand\theequation{\thesection.\arabic{equation}}
\@addtoreset{equation}{section}
\makeatother%
%EndExpansion
\end{equation}	
\end{proposition}
\begin{proof}[Proof of Proposition~\ref{Pro1}]  
\bigskip
Using the triangular inequality, we have:
\begin{equation*}
\underset{x\in
%TCIMACRO{\U{211d} }%
%BeginExpansion
\mathbb{R}
%EndExpansion
^{p}}{sup}|f_{n}(x)-f_{d}(x)|\leq \underset{x\in
%TCIMACRO{\U{211d} }%
%BeginExpansion
\mathbb{R}
%EndExpansion
^{p}}{\text{ }sup}|f_{n}(x)-\tilde{f}_{n}(x)|+\underset{x\in
%TCIMACRO{\U{211d} }%
%BeginExpansion
\mathbb{R}
%EndExpansion
^{p}}{\text{ }sup}|\tilde{f}_{n}(x)-E(\tilde{f}_{n}(x))|+\underset{x\in
%TCIMACRO{\U{211d} }%
%BeginExpansion
\mathbb{R}
%EndExpansion
^{p}}{\text{ }sup}|E(\tilde{f}_{n}(x))-f_{d}(x)|.
\end{equation*}
Now, we will study the almost sure convergence of every term.\\
For the first term
$(i)$    $\underset{x\in
%TCIMACRO{\U{211d} }%
%BeginExpansion
\mathbb{R}
%EndExpansion
^{p}}{\text{ }sup}|f_{n}(x)-\tilde{f}_{n}(x)|\overset{a.s}{\rightarrow }0$,
	
\begin{equation*}
\underset{x\in
%TCIMACRO{\U{211d} }%
%BeginExpansion
\mathbb{R}
%EndExpansion
^{p}}{\text{ }sup}|f_{n}(x)-\tilde{f}_{n}(x)|=\frac{1}{nh_{n}^{p+1}}\overset{n^{^{\prime }}-1}{\underset{m=0}{\sum }}%
K|\epsilon _{m}-\hat{\epsilon}_{m}|.
\end{equation*}
Denoted by $\varepsilon_{m}-\hat{\varepsilon}_{m}=Z_{m}$, such that
\begin{equation}
\begin{pmatrix}
\overset{\infty }{\underset{j=n+1}{\sum }}\pi _{j}^{1} &  & 0 \\
& \ddots  & \vdots  \\
0 & \cdots  & \overset{\infty }{\underset{j=n+1}{\sum }}\pi _{j}^{p}%
\end{pmatrix}%
\left(
\begin{array}{c}
X_{1+pm-j} \\
X_{2+pm-j} \\
. \\
. \\
X_{p+pm-j}%
\end{array}%
\right)
=\left(
\begin{array}{c}
z_{1+pm} \\
z_{2+pm} \\
. \\
. \\
z_{p+pm}%
\end{array}%
\right). \ \ \ \ \ \ \ \ \ \ \ \ \ \ \ \ \ \ \ \ \ \ \ \ \ \ \ \ \ \ \ \ \ \ \ \ \ \ \ \ \ \ \ \ \ \ \ \ \ \ \ \ \ \ \ \ \ \ \ \
%TCIMACRO{%
%\TeXButton{TeX field}{\makeatletter
%\renewcommand\theequation{\thesection.\arabic{equation}}
%\@addtoreset{equation}{section}
%\makeatother}}%
%BeginExpansion
\makeatletter
\renewcommand\theequation{\thesection.\arabic{equation}}
\@addtoreset{equation}{section}
\makeatother%
%EndExpansion
\end{equation}
\\
By Assumption \ref{Ass2}
\[\underset{x\in
%TCIMACRO{\U{211d} }%
%BeginExpansion
\mathbb{R}
%EndExpansion
^{p}}{sup}|f_{n}(x)-\tilde{f}_{n}(x)|\leq \frac{c}{nh_{n}^{p+1}}\overset{n^{^{\prime }}-1}{\underset{m=0}{\sum }}%
|Z_{m}|,
\]
with $c>0$
\begin{eqnarray*}
E\left( \frac{1}{nh_{n}^{p+1}}\overset{n^{^{\prime }}-1}{\underset{m=0}{\sum }}||Z_{m}||\right) ^{2} &=&\frac{1}{n^{2}h_{n}^{2p+2}}E\left( \overset{n^{^{\prime }}-1}{\underset{m=0}{\sum }}\left( ||Z_{m}||\right) ^{2}+%
\overset{n^{^{\prime }}-1}{\underset{%
\underset{t\neq m}{t=0}}{\sum }(}||Z_{m}||)(||Z_{t}||)\right)  \\
&\leq &\frac{1}{n^{2}h_{n}^{2p+2}}\left( 2
\overset{n^{^{\prime }}-1}{\underset{m=0}{\sum }}E\left( ||Z_{m}||\right)
^{2}+\underset{\underset{t\neq m}{t=0}}{%
\overset{n^{^{\prime }}-1}{\sum }}E(||Z_{t}||)^{2}\right) .
\end{eqnarray*}
The coefficients $\pi _{j}^{i}$ for $i=1,...,p$ are square summable, (see \cite{O1993} for the fixed $i$).\\ 
The almost sure convergence of $(i)$ is then verified.\\
Next, for the second term\\
$(ii)$    $\underset{x\in
%TCIMACRO{\U{211d} }%
%BeginExpansion
\mathbb{R}
%EndExpansion
^{p}}{\sup }|\overset{\sim }{f}_{n}(x)-E(\overset{\sim }{f}_{n}(x))|\overset{a.s}{\rightarrow }0$.\\
Let us the Prakasa Rao \cite{P1983} inequality,  modified for the case of multivariate stationary process (see \cite{MH2018}).
Posing
\begin{equation*}
\delta _{m}(x,\epsilon _{m})=\frac{1}{h_{n}^{p}}K\left( \frac{%
x-\epsilon _{m}}{h_{n}}\right),
\end{equation*}
and  $s_{n}=nh_{n}$, there exists $c_{0}>0$ such that	
\begin{equation*}
\underset{x\in
%TCIMACRO{\U{211d} }%
%BeginExpansion
\mathbb{R}
%EndExpansion
^{p}}{\sup }\delta _{m}(x,\epsilon _{m})\leq c_{0}s_{n}.
\end{equation*}
	
\begin{equation*}
%TCIMACRO{\U{2124} }%
%BeginExpansion
\mathbb{P}
%EndExpansion
\left( |\overset{\sim }{f}_{n}(x)-E(\overset{\sim }{f}_{n}(x))|>\epsilon
\sqrt{\frac{s_{n}\log (n)}{n}}\right) \leq 2\exp \left( -\frac{s_{n}\log(n)\epsilon ^{2}}{8c_{0}M}\right),
\end{equation*}
	
\begin{equation*}
%TCIMACRO{\U{2124} }%
%BeginExpansion
\mathbb{P}
%EndExpansion
\left( |\overset{\sim }{f}_{n}(x)-E(\overset{\sim }{f}_{n}(x))|>\epsilon n^{%
\frac{\alpha }{2}}\log (n)\right) \leq 2\exp \left( -\frac{n^{\alpha +1}\log
^{2}(n)\epsilon ^{2}}{8c_{0}M}\right),
\end{equation*}%
\begin{equation*}
%TCIMACRO{\U{2124} }%
%BeginExpansion
\mathbb{P}
%EndExpansion
\left( n^{1/4}\underset{x\in
%TCIMACRO{\U{211d} }%
%BeginExpansion
\mathbb{R}
%EndExpansion
^{p}}{\sup }|\overset{\sim }{f}_{n}(x)-E(\overset{\sim }{f}%
_{n}(x))|>\epsilon n^{\frac{2\alpha +1}{4}}\log (n)\right) \leq 2\exp \left(
-\frac{n^{\frac{4\alpha +5}{4}}\log ^{2}(n)\epsilon ^{2}}{8c_{0}M}\right),
\end{equation*}
Hence $\underset{n\rightarrow \infty }{lim}\frac{%
s_{n}\log (n)}{n}=0.$
Moreover, let $a_{n}$ be the sequence such that $a_{n}=\beta \log (n),$ where $\beta \geq 2$.
Since $0<\frac{\epsilon ^{2}n^{\frac{4\alpha +5}{4}}\log ^{2}(n)}{8c_{0}M}%
<\infty, $ we have exp$\left( -\frac{\epsilon ^{2}n^{\frac{4\alpha +5}{4}}\log ^{2}(n)}{8c_{0}M}%
\right) <n^{\beta }.$
\begin{equation*}
%TCIMACRO{\U{2124} }%
%BeginExpansion
\mathbb{P}
%EndExpansion
\left( n^{1/4}\underset{x\in
%TCIMACRO{\U{211d} }%
%BeginExpansion
\mathbb{R}
%EndExpansion
^{p}}{\sup }|\overset{\sim }{f}_{n}(x)-E(\overset{\sim }{f}%
_{n}(x))|>\epsilon n^{\frac{2\alpha +1}{4}}\log (n)\right) \leq \frac{2}{%
n^{\beta }},
\end{equation*}
\begin{equation*}
\underset{n>1}{\sum }%
%TCIMACRO{\U{2124} }%
%BeginExpansion
\mathbb{P}
%EndExpansion
\left( n^{1/4}\underset{x\in
%TCIMACRO{\U{211d} }%
%BeginExpansion
\mathbb{R}
%EndExpansion
^{p}}{\sup }|\overset{\sim }{f}_{n}(x)-E(\overset{\sim }{f}%
_{n}(x))|>\epsilon n^{\frac{2\alpha +1}{4}}\log (n)\right) \leq \underset{n>1%
}{\sum }\frac{2}{n^{\beta }}.
\end{equation*}
Since the serie	$\underset{n>1}{\sum }\frac{2}{n^{\beta }}$ converges, by Borel-cantelli
Lemma, it follows that
\begin{equation}
\underset{x\in
%TCIMACRO{\U{211d} }%
%BeginExpansion
\mathbb{R}
%EndExpansion
^{p}}{\sup }|\overset{\sim }{f}_{n}(x)-E(\overset{\sim }{f}%
_{n}(x))|>\epsilon =o\left( n^{\frac{2\alpha +1}{4}}\log (n)\right),
\end{equation}
almost surely when $n\rightarrow \infty .$ Hence $\underset{x\in
		%TCIMACRO{\U{211d} }%
		%BeginExpansion
		\mathbb{R}
		%EndExpansion
		^{p}}{\sup }|\overset{\sim }{f}_{n}(x)-E(\overset{\sim }{f}_{n}(x))|$
	converges a.s to 0.\\
	Finally, for the third term\\
	$(iii)$ $\underset{x\in
		%TCIMACRO{\U{211d} }%
		%BeginExpansion
		\mathbb{R}
		%EndExpansion
		^{p}}{\text{ }sup}|E(\tilde{f}_{n}(x))-f_{d}(x)|\overset{a.s}{\rightarrow }0$.\\
	We have        
	
	\begin{eqnarray*}
		E\left( \overset{\sim }{f}_{n}(x)\right)  &=&\frac{1}{h_{n}^{p}}E\left(
		K\left( \frac{x-\epsilon _{1}}{h_{n}}\right) \right)  \\
		&=&\frac{1}{h_{n}^{p}}\int_{%
			%TCIMACRO{\U{211d} }%
			%BeginExpansion
			\mathbb{R}
			%EndExpansion
			^{p}}K\left( \frac{x-z}{h_{n}}\right) f_{d}(z)dx \\
		&=&\frac{1}{h_{n}^{p}}\int_{%
			%TCIMACRO{\U{211d} }%
			%BeginExpansion
			\mathbb{R}
			%EndExpansion
			^{p}}K\left( u\right) f_{d}(x-uh_{n})du,
	\end{eqnarray*}
	
	by a simple calculus and using the generalized developement of Taylor we get
	
	\begin{eqnarray*}
		E\left( \overset{\sim }{f}_{n}(x)\right) -f_{d}(x) &=&\int_{%
			%TCIMACRO{\U{211d} }%
			%BeginExpansion
			\mathbb{R}
			%EndExpansion
			^{p}}K(u)\left[|\underset{k=1}{\overset{p}{\sum }}\frac{\partial f_{d}}{\partial x_{k}}%
		(x)|(-h_{n})u_{k}+ \frac{h_{n}^{2}}{2}|\overset{p}{\underset{j=1}{\sum }}%
		\overset{p}{\underset{k=1}{\sum }}\frac{\partial ^{2}f_{d}}{\partial
			x_{j}x_{k}}(x)|u_{j}u_{k}+o(h_{n}^{2})\right]du \\
		&=&\int_{%
			%TCIMACRO{\U{211d} }%
			%BeginExpansion
			\mathbb{R}
			%EndExpansion
			^{p}}K(u)\left[ \frac{h_{n}^{2}}{2}|\overset{p}{\underset{k=1}{\sum }}\frac{%
			\partial ^{2}f_{d}}{\partial x_{k}^{2}}(x)|u_{k}^{2}+o(h_{n}^{2})\right] du
	\end{eqnarray*}
	
	\begin{equation*}
	\underset{x\in
		%TCIMACRO{\U{211d} }%
		%BeginExpansion
		\mathbb{R}
		%EndExpansion
		^{P}}{sup}|E\left( \overset{\sim }{f}_{n}(x)\right) -f_{d}(x)|\leq \frac{%
		h_{n}^{2}}{2}\overset{p}{\underset{k=1}{\sum }}	\underset{x\in
		%TCIMACRO{\U{211d} }%
		%BeginExpansion
		\mathbb{R}
		%EndExpansion
		^{P}}{sup}|\frac{\partial ^{2}f_{d}}{%
		\partial x_{k}^{2}}(x)|\int_{%
		%TCIMACRO{\U{211d} }%
		%BeginExpansion
		\mathbb{R}
		%EndExpansion
		^{p}}K(u)\left[ u_{k}^{2}+o(1)\right] du.
	\end{equation*}
	
	By the assumptions \ref{Ass2} and \ref{Ass5} we have
	
	\begin{equation*}
	\underset{n\rightarrow \infty }{lim}\frac{h_{n}^{2}}{2}\overset{p}{\underset{%
			k=1}{\sum }}	\underset{x\in
		%TCIMACRO{\U{211d} }%
		%BeginExpansion
		\mathbb{R}
		%EndExpansion
		^{P}}{sup}|\frac{\partial ^{2}f_{d}}{\partial x_{k}^{2}}(x)|\rightarrow
	0,\int_{%
		%TCIMACRO{\U{211d} }%
		%BeginExpansion
		\mathbb{R}
		%EndExpansion
		^{p}}K(u)\left[ u_{k}^{2}+o(1)\right] du<\infty , \ for\ every\ x\in
	%TCIMACRO{\U{211d} }%
	%BeginExpansion
	\mathbb{R}
	%EndExpansion
	^{p}.
	\end{equation*}
	Therefore,
	$\underset{x\in
		%TCIMACRO{\U{211d} }%
		%BeginExpansion
		\mathbb{R}
		%EndExpansion
		^{P}}{sup}|E\left( \overset{\sim }{f}_{n}(x)\right) -f_{d}(x)|\overset{}{%
		\rightarrow }0,$ as $n\rightarrow \infty .$
		
According to $(i)$, $(ii)$ and $(iii)$, we get the almost sure convergence to zero of $f_{n}(x)-f_{d}(x)$.\\ This completes the proof of Proposition\ref{Pro1}.
\end{proof}

\begin{proposition}
\label{Pro2}
Noting by $\mathbb{F}$ the set of all densities with respect to the Lebesgue measure, then for every g $\in\ \mathbb{F}$, the functional\ $T : \mathbb{F} \rightarrow \Theta$ is such that:
\begin{equation*}
T(g)=\left\{ d_{0}\in \Theta:H_{2}(g,f_{d_{0}})=\underset{d\in \Theta }{\min }%
H_{2}(g,f_{d}\}\right\}.
\end{equation*}
If such a minimum exists. In case $T(g)$ is not unique, $T(g)$ will mean one of the minimum values selected arbitrarily.  	
\end{proposition}
\begin{proof}[Proof of Proposition~\ref{Pro2}]
The proof is well detailed in $\cite{B1977}$ by Theorem 1 and in $\cite{H1995}$ by Lemma 3.1.
\end{proof}
\begin{proposition}
\label{Pro3}
Assume that the assumptions \ref{Ass3}-\ref{Ass6} and the conditions \ref{c1}-\ref{c2} hold and that $d$
lies in interior of $\Theta $. So, for any sequence $f_{n}$
converging to $f_{d}$ in the Hellinger metric, we have
\begin{equation}
T(f_{n}(n))=d+\int_{%
%TCIMACRO{\U{211d} }%
%BeginExpansion
\mathbb{R}
%EndExpansion
^{p}}U_{d}(x)\left[ f_{n}^{\frac{1}{2}}(x)-f_{d}^{\frac{1}{2}}(x)\right]
dx+V_{n}\int_{%
%TCIMACRO{\U{211d} }%
%BeginExpansion
\mathbb{R}
%EndExpansion
^{p}}g_{d}^{\prime }(x)\left[ f_{n}^{\frac{1}{2}}(x)-f_{d}^{\frac{1}{2}}(x)%
\right] dx.
\end{equation}
Here $V_{n}$ is a non-singular $p\times p$-matrix, such that the components of $%
\sqrt{n}V_{n}$ tend to zero when $n$ $\rightarrow \infty $.
\end{proposition}
\begin{proof}[Proof of Proposition~\ref{Pro3}]
See the proof of Theorem 4 in $\cite{B1977}$.
\end{proof}
\begin{proposition}
\label{Pro4}
Assume that the assumptions \ref{Ass2}-\ref{Ass3} and \ref{Ass5}-\ref{Ass7} hold. Then, the limiting distribution of $%
\sqrt{nh_{n}}\left[ f_{n}^{\frac{1}{2}}(x)-f_{d}^{\frac{1}{2}}(x)\right] $ is $N(0,f_{d}(x)\int_{%
%TCIMACRO{\U{211d} }%
%BeginExpansion
\mathbb{R}
%EndExpansion
^{p}}K^{2}(u)du).$
\end{proposition}
\begin{proof}[Proof of Proposition~\ref{Pro4}]
The proof of this is similar to the proof of Theorem 3 in $\cite{WM2002}$.
\end{proof}

After these preliminary results, we are ready to establish the MHD estimation. The almost sure convergence of $\hat{d}_{n}$ to $d$ and
it asymptotic normality are formally stated in the next section (\ref{sec2}).

\section{Asymptotic properties of the MHDE for PtvARFIMA}
\label{sec2}
\begin{theorem}
	\label{Teo1}
	Suppose that Assumptions \ref{Ass1}-\ref{Ass7} hold. Then,
	\begin{equation}
	\label{Eq6}
	\hat{d_{n}} \overset{a.s}{\rightarrow }d,\ as\ n \rightarrow \infty.
	%TCIMACRO{%
	%\TeXButton{TeX field}{\makeatletter
	%\renewcommand\theequation{\thesection.\arabic{equation}}
	%\@addtoreset{equation}{section}
	%\makeatother}}%
	%BeginExpansion
	\makeatletter
	\renewcommand\theequation{\thesection.\arabic{equation}}
	\@addtoreset{equation}{section}
	\makeatother%
	%EndExpansion
	\end{equation}
	
	.

\end{theorem}
\begin{proof}[Proof of Theorem~\ref{Teo1}]
	
	By the propositions~\ref{Pro1} and~\ref{Pro2}, the proof can directly be achieved. From proposition~\ref{Pro1}, we have
	\begin{equation*}
	%TCIMACRO{\U{2124} }%
	%BeginExpansion
	\mathbb{P}
	%EndExpansion
	\left\{ \underset{n\rightarrow \infty }{\lim }f_{n}^{\frac{1}{2}}(x)=f_{d}^{%
		\frac{1}{2}}(x)\text{ }\forall x\right\} =1.
	\end{equation*}
	
	Consequently
	
	\begin{equation*}
	H_{2}(f_{n},f_{d})\overset{a.s}{\rightarrow }0\text{ as }n\rightarrow \infty.
	\end{equation*}
	
	Next, by proposition \ref{Pro2}, $T (f_{d}) = d$ uniquely on $\Theta$, then the functional $T$ is
	continuous at $f_{d}$ in the Hellinger topology. Therefore
	\begin{equation*}
	\hat{d_{n}} = T (f_{n}(x)) \rightarrow  T (f_{d}(x)) = d,
	\end{equation*}
	almost surely as $n \rightarrow \infty$.
\end{proof}
	\begin{theorem}
		\label{Teo2}
		Suppose that the assumptions \ref{Ass1}-\ref{Ass7} and conditions \ref{c1} and \ref{c2} hold. Then, the limit
		distribution of $\sqrt{n}(\hat{d}_{n}-d)$ is
		\begin{equation}
		\sqrt{n}(\hat{d}_{n}-d)\rightarrow N(0,\Sigma ^{2}),
		\end{equation}
		
		where
		
		\begin{equation}
		\Sigma ^{2}=\frac{1}{4}\left[ \int_{%
			%TCIMACRO{\U{211d} }%
			%BeginExpansion
			\mathbb{R}
			%EndExpansion
			^{p}}g_{d}^{\prime }(x)\left[ g_{d}^{\prime }(x)\right] ^{t}dx\right]
		^{-1}\int_{%
			%TCIMACRO{\U{211d} }%
			%BeginExpansion
			\mathbb{R}
			%EndExpansion
			^{p}}k^{2}(u)du.
		\end{equation}
	\end{theorem}

	\begin{proof}[Proof of Theorem~\ref{Teo2}]
		By proposition \ref{Pro3}, one can show that

		\bigskip
		\begin{eqnarray*}
			\sqrt{n}(\hat{d}_{n}-d) &=&\sqrt{n}\int_{%
				%TCIMACRO{\U{211d} }%
				%BeginExpansion
				\mathbb{R}
				%EndExpansion
				^{p}}U_{d}(x)\left[ f_{n}^{\frac{1}{2}}(x)-f_{d}^{\frac{1}{2}}(x)\right] dx+%
			\sqrt{n}V_{n}\int_{%
				%TCIMACRO{\U{211d} }%
				%BeginExpansion
				\mathbb{R}
				%EndExpansion
				^{p}}g_{d}^{\prime }(x)\left[ f_{n}^{\frac{1}{2}}(x)-f_{d}^{\frac{1}{2}}(x)%
			\right] dx. \\
			&=&\sqrt{n}\int_{%
				%TCIMACRO{\U{211d} }%
				%BeginExpansion
				\mathbb{R}
				%EndExpansion
				^{p}}U_{d}(x)\left[ f_{n}^{\frac{1}{2}}(x)-f_{d}^{\frac{1}{2}}(x)\right]
			dx+o_{p}(1).
		\end{eqnarray*}
		
		With $V_{n}\rightarrow 0$ in probability, $U_{d}\in L_{2}$ and $U_{d}\bot f_{d}^{\frac{1}{2}%
		}$, where $\bot $ is the orthogonality in $L_{2}.$
		
		For $b\geq 0,$ $a>0,$ the algebraic identity is given by
		
		\begin{equation*}
		b^{\frac{1}{2}}-a^{\frac{1}{2}}=\frac{b-a}{2a^{\frac{1}{2}}}-\frac{(b-a)^{2}%
		}{\left[ 2a^{\frac{1}{2}}(b^{\frac{1}{2}}+a^{\frac{1}{2}})^{2}\right] }.
		\end{equation*}
		
		According to assumption \ref{Ass3}, the condition $f_{d}^{\frac{1}{2}}(x)$$>0$ and the algebraic identity, we have
		
		\begin{equation*}
		\sqrt{n}(\hat{d}_{n}-d)=\sqrt{n}\int_{%
			%TCIMACRO{\U{211d} }%
			%BeginExpansion
			\mathbb{R}
			%EndExpansion
			^{p}}U_{d}(x)\left[ \frac{f_{n}(x)-f_{d}(x)}{2f_{d}^{\frac{1}{2}}(x)}\right]
		dx+A_{n}.
		\end{equation*}
		
		Where
		\begin{equation*}
		A_{n}=-\sqrt{n}\int_{%
			%TCIMACRO{\U{211d} }%
			%BeginExpansion
			\mathbb{R}
			%EndExpansion
			^{p}}U_{d}(x)\left[ \frac{\left[ f_{n}(x)-f_{d}(x)\right] ^{2}}{2f_{d}^{%
				\frac{1}{2}}(x)(f_{n}^{\frac{1}{2}}(x)+f_{d}^{\frac{1}{2}}(x))^{2}}\right] dx.
		\end{equation*}
		
		Hence
		
		\begin{equation*}
		|A_{n}|\leq 2\delta ^{-\frac{3}{2}}\int_{%
			%TCIMACRO{\U{211d} }%
			%BeginExpansion
			\mathbb{R}
			%EndExpansion
			^{p}}|U_{d}(x)|\sqrt{n}\left[ f_{n}(x)-f_{d}(x)\right] ^{2},
		\end{equation*}
		
		where $\delta =\underset{x\in
			%TCIMACRO{\U{211d} }%
			%BeginExpansion
			\mathbb{R}
			%EndExpansion
			^{p}}{\inf }f(x),$
		\begin{center}
			$2f_{d}^{\frac{1}{2}}(x)(f_{n}^{\frac{1}{2}}(x)+f_{d}^{%
				\frac{1}{2}}(x))^{2}>2\delta ^{\frac{3}{2}}.$
		\end{center}

		Conditions \ref{c1} and \ref{c2} imply that $U_{d}(x)$ is continuous and bounded. So, by proposition 2.1 and the Vitali's Theorem, $|A_{n}|$ tends to $0$ in probability for $n\rightarrow \infty .$
		
		So, we can rewrite $\sqrt{n}(\hat{d}_{n}-d)$ as folows
		
		\begin{equation*}
		\sqrt{n}(\hat{d}_{n}-d)=\sqrt{n}\int_{%
			%TCIMACRO{\U{211d} }%
			%BeginExpansion
			\mathbb{R}
			%EndExpansion
			^{p}}U_{d}(x)\left[ \frac{f_{n}(x)-f_{d}(x)}{2f_{d}^{\frac{1}{2}}(x)}\right]dx+o_{p}(1).
		\end{equation*}
		
		By proposition \ref{Pro4} and by a simple calculus we deduce the limiting distribution of $\sqrt{n}(\hat{d}_{n}-d)$,
		
		\begin{eqnarray*}
			&&\int_{%
				%TCIMACRO{\U{211d} }%
				%BeginExpansion
				\mathbb{R}
				%EndExpansion
				^{p}}\left[ \frac{U_{d}(x)}{2f_{d}^{\frac{1}{2}}(x)}\right] \left[ \frac{%
				U_{d}(x)}{2f_{d}^{\frac{1}{2}}(x)}\right] ^{t}\int_{%
				%TCIMACRO{\U{211d} }%
				%BeginExpansion
				\mathbb{R}
				%EndExpansion
				^{p}}K^{2}(u)duf_{d}(x)dx \\
			&=&\frac{1}{4}\int_{%
				%TCIMACRO{\U{211d} }%
				%BeginExpansion
				\mathbb{R}
				%EndExpansion
				^{p}}U_{d}(x)(U_{d}(x))^{t}dx\int_{%
				%TCIMACRO{\U{211d} }%
				%BeginExpansion
				\mathbb{R}
				%EndExpansion
				^{p}}K^{2}(u)du.
		\end{eqnarray*}
		
		Which proves Theorem \ref{Teo2}.
	\end{proof}
	\section{Simulation}
	\label{sec3}   
	In this section, we present a simulation experiment, in order to illustrate, the most significative results for the method proposed in this article, we apply numerically the Minimum Hellinger distance method. We consider two cases:\\ First case, we assume that the white noise $\epsilon _{m}$, $\epsilon _{m}=(\varepsilon_{1+2m},\varepsilon_{2+2m})^{^{\prime }}$  is the
	density function of the standard normal distribution and in this case the kernel density $K$ is also the density of 
	the standard normal distribution. \\Second case, we consider $\epsilon _{m}$ and K
	as the density function of the Cauchy distribution (see \cite{LPK2014}).\\ We generate a
	PtvARFIMA$_{2}(0,d,0),$ $X_{m}=(X_{1+2m},X_{2+2m})^{^{\prime }}$, $%
	m=0,1,..,n\prime -1,$ we use three simple size $n=10;\ 50;\ 100$, with two
	different values of $d,$ such that $d=(0.2,0.15)$ and $d=(0.49,0.4).$ We have

	\begin{equation*}
	MSE=\frac{1}{n_{r}}\underset{j=1}{\overset{n_{r}}{\sum }}\{(\hat{d}%
	_{n,1}^{j}-d_{1})^{2}+(\hat{d}_{n,2}^{j}-d_{2})^{2}\},
	\end{equation*}
	
	where $n_{r}$ $=100$ is the number of replications and $(\hat{d}_{n,1}^{j},%
	\hat{d}_{n,2}^{j})$ denote the estimate of $d$ for the $j$th replication.
	The results are displayed in the following tables:\\
	
	For the first case, 
	\begin{table}[h!]
		\centering
		
		\begin{tabular}{|l|l|l|l|}
			\hline
			$n$ & $10$ & $50$ & $100$ \\ \hline
			$MHDE$ & $(0.2006,0.1494)$ & $(0.1987,0.1512)$ & $(0.2077,0.1423)$ \\ \hline
			
			$MSE$ & $0.0075$ & $0.0071$ & $0.0054$ \\ \hline
		\end{tabular}\\
		\caption{MHDE for $d=(0.2,0.15)$. }
		\label{Tab1}
	\end{table}

	\begin{table}[h!]
		\centering
		
		\begin{tabular}{|l|l|l|l|}
			\hline
			$n$ & $10$ & $50$ & $100$ \\ \hline
			$MHDE$ & $(0.4944,0.0.3956)$ & $(0.4939,0.4011)$ & $(0.4938,0.401)$ \\ \hline
			
			$MSE$ & $0.000895$ & $0.000738$ & $0.00064$ \\ \hline
		\end{tabular}\\
		
		\caption{MHDE for $d=(0.49,0.4)$. }
		\label{Tab2}
		
	\end{table}

	Second case
	
	\begin{table}[h!]
		\centering
		
		\begin{tabular}{|l|l|l|l|}
			\hline
			$n$ & $10$ & $50$ & $100$ \\ \hline
			$MHDE$ & $(0.2084,0.1416)$ & $(0.1992,0.1507)$ & $(0.1988,0.1511)$ \\ \hline
			
			$MSE$ & $0.0084$ & $0.0068$ & $0.0063$ \\ \hline
		\end{tabular}\\
		\caption{MHDE for $d=(0.2,0.15)$. }
		\label{Tab3}
		
	\end{table}
	
	\begin{table}[h!]
		\centering
		\begin{tabular}{|l|l|l|l|}
			\hline
			$n$ & $10$ & $50$ & $100$ \\ \hline
			$MHDE$ & $(0.4944,0.3952)$ & $(0.494,0.3992)$ & $(0.4943,0.3973)$ \\ \hline
			$MSE$ & $0.000859$ & $0.000796$ & $0.000758$ \\ \hline
		\end{tabular}\\
		\caption{MHDE for $d=(0.49,0.4)$. }
		\label{Tab4}
	\end{table}
	
	From tables (Table \ref{Tab1}, \ref{Tab2}, \ref{Tab3} and \ref{Tab4}) we can deduce that the mean parameter estimates are very close in value to the true value of the parameters, and the normal law provides a high accuracy. However, the MSE is decreasing with the increase in sample size. This happens because the impact of the decrease in variance with increasing sample size.  
	
	\section{conclusion}
	In this paper, we have addressed the problem of the Minimum Hellinger distance (MHD) in the \cite{B1977}'s style for a periodically time-varying long-memory parameter. We have constructed an estimate for the vector parameters $d$, using the MHD method and studied its asymptotic properties by considering the related multivariate stationary model. We have also presented some numerical simulations illustrating our theoretical results.

\end{document}